\documentclass[letterpaper,12pt]{amsart}

\usepackage[colorlinks, linktocpage, allcolors=black,breaklinks]{hyperref}
\usepackage{amsmath,amssymb,amsthm,mathrsfs,bbm,comment,units,enumitem,xcolor}

\renewcommand{\leq}{\leqslant}

\newtheorem{theorem}{Theorem}	%%[section] <-- adds the section numbering e.g. Thm 0.4
\newtheorem{lemma}[theorem]{Lemma}
\newtheorem{proposition}[theorem]{Proposition}
\newtheorem{corollary}[theorem]{Corollary}
\newtheorem{definition}[theorem]{Definition}

\newcommand{\astfill}{\noindent\xleaders\hbox{$\ast$}\hfill\kern0pt}
\newcommand{\closure}{\mathrm{cl}}
\newcommand{\complexes}{\mathbb{C}}

\newcommand{\diam}{\mathrm{diam}}
\newcommand{\diameter}{\diam}

\newcommand{\integers}{\mathbb{Z}}
\newcommand{\interior}{\text{interior}}
\newcommand{\naturals}{\mathbb{N}}
\newcommand{\rationals}{\mathbb{Q}}
\newcommand{\reals}{\mathbb{R}}
\newcommand{\symdif}{\bigtriangleup}

\begin{document}

\title{An Application of Descriptive Set Theory to Complex Analysis}
\author{Christopher J. Caruvana}
\address{Indiana University Kokomo School of Sciences\\
2300 S Washington St, Kokomo, Indiana 46902, USA}
\email{chcaru@iu.edu}
\urladdr{https://chcaru.pages.iu.edu/}

\author{Robert R. Kallman}
\address{Denton, Texas, USA}
\email{kallman.robert@gmail.com}

\date{\today}


\maketitle

%%%%%%%%%%%%%%%%%%%%%%%%%%%%%%%%%%%%%%%%%%%%%%%%%%%%%%%%%%%%%%%%%%%%%%%%%%%%%%%%%%%%%%%%%%%%%%%%%%%%%%%%%%%%%
%%%%%%%%%%%%%%%%%%%%%%%%%%%%%%%%%%%%%%%%%%%%%%%%%%%%%%%%%%%%%%%%%%%%%%%%%%%%%%%%%%%%%%%%%%%%%%%%%%%%%%%%%%%%%

\begin{abstract}
The purpose of this paper is to prove a new general result about rings of complex analytic functions.   Let
$\Omega$ be an arbitrary nonempty open subset of the complex plane $\complexes$, $\mathcal{A}(\Omega)$ be the
set of holomorphic functions on $\Omega$ viewed as a Polish ring (not a Polish algebra over $\complexes$) in
the usual compact open topology, let $R$ be a Polish ring and let $\varphi : R \to \mathcal{A}(\Omega)$ be
an abstract algebraic isomorphism.  The main goal of this paper is to prove Theorem
\ref{theorem:ring-complete} that $\varphi$ is a topological isomorphism.  A special result of Bers is an easy
corollary.  Two additional items supplement these results, viz., that $B(\mathbb{D})$, the abstract ring of
bounded analytic functions on the unit disk, cannot be made into a Polish ring and that $\mathcal{M}(\Omega)$,
the abstract field of meromorphic functions on $\Omega$, cannot be made into a Polish field.
\end{abstract}

%%%%%%%%%%%%%%%%%%%%%%%%%%%%%%%%%%%%%%%%%%%%%%%%%%%%%%%%%%%%%%%%%%%%%%%%%%%%%%%%%%%%%%%%%%%%%%%%%%%%%%%%%%%%%
%%%%%%%%%%%%%%%%%%%%%%%%%%%%%%%%%%%%%%%%%%%%%%%%%%%%%%%%%%%%%%%%%%%%%%%%%%%%%%%%%%%%%%%%%%%%%%%%%%%%%%%%%%%%%

\section{Introduction}  \label{section:introduction}

The purpose of this paper is to prove a new general result about rings of complex analytic functions.   Let
$\Omega$ be an arbitrary nonempty open subset of the complex plane $\complexes$, $\mathcal{A}(\Omega)$ be the
set of holomorphic functions on $\Omega$ viewed as a Polish ring (not a Polish algebra over $\complexes$) in
the usual compact open topology, let $R$ be a Polish ring and let $\varphi : R \to \mathcal{A}(\Omega)$ be
an abstract algebraic isomorphism.  The main goal of this paper is to prove Theorem
\ref{theorem:ring-complete}, that $\varphi$ is a topological isomorphism.  A special result of Bers is an easy
corollary.  Two additional items are proved that helps put this theorem into perspective, that
$B(\mathbb{D})$, the abstract ring of bounded analytic functions on the unit disk, cannot be made into a
Polish ring and that $\mathcal{M}(\Omega)$, the abstract field of meromorphic functions on $\Omega$, cannot be
made into a Polish field.

This paper assumes a familiarity with descriptive set theory, especially Polish spaces, analytic sets, sets
with the Baire property and Borel spaces, as may be found in Parthasarathy (\cite{parthasarathy-1967a}),
Kechris (\cite{kechris-1995a}), Becker and Kechris (\cite{becker-kechris-1996a}) and Mackey
(\cite{mackey-1957a}).

The by now well known very general approach to the proof of theorems like Theorem \ref{theorem:ring-complete}
is to show that $\varphi$ is measurable with respect to the sets with the Baire property.  However, in almost
every proof of such theorems the verification of the measurability condition in any particular instance can be
intricate and require a fair amount of ingenuity.  Such seems to be the case here because of the generality of
the result and that a priori there appears so little structure to work with.

A complicating factor in the proof is that $\mathcal{A}(\Omega)$ is viewed as a Polish ring and not as a
Polish algebra over $\complexes$.  The proof is largely carried out in a sequence of lemmas and propositions,
some of which may be of independent interest. A number of the propositions in the following discussion are
given without proofs being presented.  In every case either references are given or the proofs are easy and
left to the reader.

We always assume unless otherwise stated that if $\Omega \subseteq \complexes$ is open then $\emptyset \ne
\Omega$.

%%%%%%%%%%%%%%%%%%%%%%%%%%%%%%%%%%%%%%%%%%%%%%%%%%%%%%%%%%%%%%%%%%%%%%%%%%%%%%%%%%%%%%%%%%%%%%%%%%%%%%%%%%%%%%
%%%%%%%%%%%%%%%%%%%%%%%%%%%%%%%%%%%%%%%%%%%%%%%%%%%%%%%%%%%%%%%%%%%%%%%%%%%%%%%%%%%%%%%%%%%%%%%%%%%%%%%%%%%%%%

\section{Borel Set Preliminaries}

The following general theorems and corollaries will prove to be useful.

\begin{theorem} [{\cite[Theorem 1.2.6]{becker-kechris-1996a}}]  \label{theorem:becker-kechris}
Let $\varphi: G \to H$ be a Baire Property measurable homomorphism between Polish groups.  Then $\varphi$
is continuous.  If moreover, $\varphi[G]$ is not meager, then $\varphi$ is also open.
\end{theorem}

\begin{theorem}[{\cite[Theorem 12.17]{kechris-1995a}}] \label{thm:ClosedSubgroupsAdmitBorelTransversal}
Let $G$ be a Polish group and $H$ a closed subgroup of $G$. Then there exist a Borel subset $B$ of $G$ so that
$B$ meets every $H$-coset at exactly one point.
\end{theorem}

\begin{proposition}[{\cite[Proposition 5]{atim-kallman-2012}}] \label{prop:Atim}
Suppose $G$ is a multiplicative Polish group, $H$ is an analytic subgroup of $G$, and $A \subseteq G$ is an
analytic subset of $G$ so that $A$ meets every $H$-coset at exactly one point and $G = AH$. Then $H$ is
closed.
\end{proposition}

\begin{corollary} \label{cor:Atim}
Let $G$ be a multiplicative Polish group, $H$ and $K$ be subgroups of $G$ so that $H$ and $K$ are analytic
sets, $G = HK$, and $H \cap K = \{e\}$. Then both $H$ and $K$ are closed.
\end{corollary}

Recall that if $X$ is a set then a family $\mathscr{F}$ of subsets of $X$ separates points if for every pair
$x \neq y \in X$, there exists $A \in \mathscr{F}$ so that $x \in A$ and $y \not\in A$.

\begin{theorem}[Mackey {\cite[Theorem 3.3]{mackey-1957a}}]  \label{theorem:Mackey}
Let $X$ be a Polish space and suppose $\mathscr{F}$ is a countable family of Borel sets which separates
points. Then the family $\mathscr{F}$ generates the Borel structure of $X$. That is, the smallest
$\sigma$-algebra containing all members of $\mathscr{F}$ is precisely $\mathscr{B}(X)$.
\end{theorem}

View $\rationals^{2} \subset \reals^{2} = \complexes$ as complex numbers.

\begin{corollary} \label{corollary:Mcor}
Let $\mathscr{F}$ be a countable family of Borel subsets of $\complexes$, each with nonempty interior and with
the property that, for every $\varepsilon > 0$, there is some $B \in \mathscr{F}$ with $\diameter(B) < \varepsilon$.
Then the countable collection $\{ q + B \ | \ q \in \rationals^{2} \text{, } B \in \mathscr{F} \}$ generates
the Borel structure of $\complexes$, i.e., generates $\mathscr{B}(\complexes)$.
\end{corollary}
\begin{proof}
This follows from Theorem \ref{theorem:Mackey} since it is easy to check that the countable collection $\{ q
+ B \ | \ q \in \rationals^{2} \text{, } B \in \mathscr{F} \}$ separates points.
\end{proof}

In what follows $\mathbb{D} = \{ \zeta \in \complexes \ | \  |\zeta| < 1 \}$ will denote the open unit disk
in the complex plane and $\closure(\mathbb{D}) = \{ \zeta \in \complexes \ | \  |\zeta| \le 1 \}$ is the
closed unit disk in the complex plane.

\begin{corollary} \label{corollary:crat-disk}
The $\sigma$-algebra generated by the sets $\{ \beta + \mathbb{D} \ | \ \beta \in \rationals^{2} \}$ is
$\mathscr{B}(\complexes)$. Similarly, the $\sigma$-algebra generated by the sets $\{ \beta +
\closure(\mathbb{D}) \ | \ \beta \in \rationals^{2} \}$ is $\mathscr{B}(\complexes)$.
\end{corollary}
\begin{proof}
Let $\delta > 0$, $\delta \in \rationals$ be very small.  Then $\mathbb{D} \cap (2 - \delta + \mathbb{D})$ is a
nonempty open set that fits inside a $\delta \times \sqrt{4\delta - \delta^{2}}$ box.  That is, first
translate $\mathbb{D}$ to be centered over $2$ and then move it slightly left.  Now use Corollary
\ref{corollary:Mcor}.  The proof for $\closure(\mathbb{D})$ is virtually the same.
\end{proof}

\begin{lemma} \label{lemma:Baire-homeomorphism}
Let $X$ be a Polish space, $\varphi$ a homeomorphism of $X$ and $B \subseteq X$ a subset with the Baire
property.  Then $\varphi[B]$ has the Baire property.
\end{lemma}
\begin{comment}
  \begin{proof}
  If $C \subseteq X$ is a closed set with empty interior, then $\varphi[C]$ is a closed set with empty interior.
  Therefore if $S \subseteq X$ is meager, then $S$ is contained in the union of a sequence closed nowhere dense
  sets, $\varphi[S]$ is also contained in the union of a sequence of closed nowhere dense sets and therefore
  $\varphi[S]$ is meager.   Choose an open set $U$ so that $B \symdif U$ is meager, so $\varphi[B] \symdif
  \varphi[U] = \varphi[B \symdif U]$ is also meager and therefore $\varphi[B]$ is a set with the Baire property.
  \end{proof}
\end{comment}

\begin{corollary} \label{corollary:ring-Baire}
If $R$ is a Polish ring, $a$, $b \in R$ with $a$ invertible, $B \subseteq R$ a subset with the Baire property, then
$aB + b$ is also a subset with the Baire property.
\end{corollary}
\begin{proof}
$\varphi(x) = ax + b$ is a homeomorphism since $a$ is invertible.  Now use Lemma \ref{lemma:Baire-homeomorphism}.
\end{proof}

Though it is consistent with ZFC that the continuous image of a set with the Baire property fails to have the
Baire property (see \cite{moschovakis} for the existence of \(\mathbf \Delta^{1}_{2}\) sets that lack the
Baire property), this corollary suggests the following question. If $R$ is a Polish ring, $a \in R$ and $B
\subseteq R$ is a set with the Baire property, does $aB$ have the Baire property?

\begin{lemma} \label{lem:LinearImageOfBall}
Let $\alpha ,\beta \in \complexes$ with $\alpha \neq 0$ and $f(\zeta) = \alpha \zeta + \beta$.
Then, for any $\zeta \in \complexes$ and $r > 0$, $f[B(\zeta,r)] = B(f(\zeta), |\alpha| r)$.
\end{lemma}

\begin{lemma} \label{lem:ExteriorSeparates}
Let $\Omega \subseteq \complexes$ be a Borel set with $\interior(\Omega) \ne \emptyset$ and $\interior(\Omega^{c}) \ne \emptyset$.
Then
\begin{displaymath}
	\mathscr{U} = \{ \alpha \Omega + \beta \ | \ \alpha, \beta \in \rationals^{2},\ \alpha \neq 0 \}
\end{displaymath}
is a countable family of Borel sets that separates points.
\end{lemma}
\begin{proof}
$\mathscr U$ is a countable family of Borel sets since Borel sets are invariant under homeomorphisms.  We must
see that $\mathscr{U}$ separates points. Let
$\zeta_{1} \in \rationals^{2} \cap \interior(\Omega)$, $\zeta_{2} \in \rationals^{2} \cap
\interior(\Omega^{c})$, and $r > 0$ be so that $B(\zeta_{1},r) \subseteq \interior(\Omega)$ and
$B(\zeta_{2},r) \subseteq \interior(\Omega^{c})$.

Now, let $w_{1}$, $w_{2} \in \complexes$ be arbitrary so that $w_{1} \neq w_{2}$ and pick $\varepsilon > 0$ so small that
\begin{displaymath}
2\varepsilon < |w_{1} - w_{2}| \text{ and } \varepsilon|\zeta_{1} - \zeta_{2}| + 2r\varepsilon < |w_{1} - w_{2}|r.
\end{displaymath}
Pick $\lambda_{j} \in B(w_{j},\varepsilon) \cap \rationals^{2}$ where $1 \le j \le 2$.   Notice that $B(w_{1},\varepsilon) \cap
B(w_{2},\varepsilon) = \emptyset$ since $2\varepsilon < |w_{1} - w_{2}|$ and therefore $0 < |\lambda_{1} - \lambda_{2}|$.
Also easily check that $|w_{1} - w_{2}| < |\lambda_{1} - \lambda_{2}| + 2\varepsilon$ by the triangle inequality.

Let $\varphi(\zeta) = \dfrac{\lambda_{2} - \lambda_{1}}{\zeta_{2} - \zeta_{1}}(\zeta -\zeta_{1}) + \lambda_{1}
= \alpha \zeta + \beta$, where $\alpha = \dfrac{\lambda_{2} - \lambda_{1}}{\zeta_{2} - \zeta_{1}}$ and $\beta =
\dfrac{\lambda_{1} \zeta_{2} - \lambda_{2} \zeta_{1}}{\zeta_{2} - \zeta_{1}}$ are both complex rationals.
Notice that $\varphi(\zeta_{1}) = \lambda_{1}$ and $\varphi(\zeta_{2}) = \lambda_{2}$. We will prove that
$\alpha \Omega + \beta$ separates $w_{1}$ and $w_{2}$.

$\varphi(\zeta_{j}) = \lambda_{j}$ and therefore $\varphi[B(\zeta_{j},r)] = B\left(\lambda_{j},\dfrac{|\lambda_{1} -
\lambda_{2}|r}{|\zeta_{1} - \zeta_{2}|}\right)$ by Lemma \ref{lem:LinearImageOfBall}.

$|w_{j} - \lambda_{j}| < \varepsilon < \dfrac{(|w_{1} - w_{2}| - 2\varepsilon)r}{|\zeta_{1} - \zeta_{2}|} <
\dfrac{|\lambda_{1} - \lambda_{2}|r}{|\zeta_{1} - \zeta_{2}|}$ and therefore $w_{j} \in
\varphi[B(\zeta_{j},r)]$, ($1 \le j \le 2$).  Hence, $w_{1} \in \alpha \Omega + \beta$ and $w_{2} \not \in
\alpha \Omega + \beta$,
\end{proof}

\begin{proposition} \label{prop:ContinuityByExterior}
Let $R$ be a Polish ring and $\varphi : R \to \complexes$ be an abstract ring isomorphism. If there exists
a Borel set $\Omega \subseteq \complexes$ with $\interior(\Omega) \ne \emptyset$ and $\interior(\Omega^{c})
\ne \emptyset$ so that $\varphi^{-1}[\Omega]$ has the Baire property, then $\varphi$ is a topological
isomorphism.
\end{proposition}
\begin{proof}
First, apply Lemma \ref{lem:ExteriorSeparates} and then Theorem \ref{theorem:Mackey} to see that $\{ \alpha
\Omega + \beta | \alpha, \beta \in \rationals^{2}, \alpha \neq 0\}$ generates the Borel structure of
$\complexes$. Then, since $\varphi^{-1}[\Omega]$ has the Baire property and $R$ is a Polish ring, we have that
$\varphi^{-1}(\alpha)\varphi^{-1}[\Omega] + \varphi^{-1}(\beta)$ is a set with the Baire property for every
$\alpha, \beta \in \complexes$, $\alpha \neq 0$ by Lemma \ref{corollary:ring-Baire}.  That is, $\varphi$ is a
$\mathscr{BP}$-measurable isomorphism so Theorem \ref{theorem:becker-kechris} applies to conclude that
$\varphi$ is a topological isomorphism.
\end{proof}

%%%%%%%%%%%%%%%%%%%%%%%%%%%%%%%%%%%%%%%%%%%%%%%%%%%%%%%%%%%%%%%%%%%%%%%%%%%%%%%%%%%%%%%%%%%%%%%%%%%%%%%%%%%%%%
%%%%%%%%%%%%%%%%%%%%%%%%%%%%%%%%%%%%%%%%%%%%%%%%%%%%%%%%%%%%%%%%%%%%%%%%%%%%%%%%%%%%%%%%%%%%%%%%%%%%%%%%%%%%%%

\section{General Polish Ring Preliminaries}

This section is devoted to some general facts about Polish rings and Polish $\complexes$-vector spaces.

\begin{proposition} \label{prop:InvertiblesAreBorel}
Let $R$ by any Polish ring with unity and let
\begin{displaymath}
	\mathcal{I}_{R} = \{ x \in R \ | \ x \text{ has a right inverse } \}.
\end{displaymath}
Then $\mathcal{I}_{R}$ is an analytic set.  $\mathcal{I}_{R}$ is a Borel set if $R$ is commutative.
\end{proposition}
\begin{proof}
Notice that
\begin{displaymath}
	\mathcal{I}_{R} = \{ x \in R \ | \text{ there exists } y \in R \text{ with } xy = 1_{R} \}.
\end{displaymath}
Let $A = \{ \langle x,y \rangle \in R^{2} \ | \ xy = 1_{R} \}$ and $\pi : R^{2} \to R$ be the projection $\langle x,y \rangle \mapsto x$.
$A$ is closed in $R^{2}$ since multiplication is continuous.  $\mathcal{I}_{R}$ is an analytic set since it is
the range of $\pi$.

Suppose $R$ is commutative. To see that $\mathcal{I}_{R}$ is Borel, it suffices to check that $\pi \restriction_A$ is
injective. To see this, suppose that $\langle x,y_{1} \rangle$, $\langle x,y_{2} \rangle \in A$. Then $xy_{1} = 1_{R} = xy_{2} \implies
y_{1} = y_{1}xy_{2} = xy_{1}y_{2} = y_{2}$.
\end{proof}

Is $\mathcal{I}_{R}$ always a Borel set?

\begin{proposition} \label{prop:ContinuityOfInversion}
Let $R$ be a commutative Polish ring with unity.
Then the map $x \mapsto x^{-1}$, $\mathcal{I}_{R} \to \mathcal{I}_R$, is continuous if and only if
$\mathcal{I}_{R}$ is a $G_{\delta}$ subset of $R$.
\end{proposition}
\begin{proof}
First assume that $x \mapsto x^{-1}$, $\mathcal{I}_{R} \to \mathcal{I}_{R}$, is continuous. Let $A = \{
\langle x,y \rangle \in R^{2} \ | \ xy = 1_{R} \}$ and notice that $A$ is closed in $R^{2}$. The canonical projection $\pi :
R^{2} \to R$ defined by $\langle x,y \rangle \mapsto x$ is injective when restricted to $A$. Since $x \mapsto x^{-1}$,
$\mathcal{I}_{R} \to \mathcal{I}_{R}$, is continuous, we see that $x \mapsto \langle x,x^{-1} \rangle$, $\mathcal{I}_{R}
\to A$, is continuous. So $\pi\restriction_A$ is a homeomorphism which implies that $\mathcal{I}_{R}$ is Polishable
and thus is a $G_{\delta}$ subset of $R$.

Conversely, suppose that $G = \mathcal{I}_{R}$ is a $G_{\delta}$ subset of $R$.  Therefore $G$ is Polishable.
Notice that the map $x\mapsto x^{-1}$, $G \to G$, is a multiplicative group isomorphism.
Let $A = \{ \langle x,y \rangle \in G^{2} \ | \ xy = 1_{R} \}$ and notice that $A$ is closed in $G^{2}$ since $(x,y) \mapsto
xy$, $G^{2} \to G$, is continuous.

The mappings $A \to G$, $\langle x,x^{-1} \rangle \mapsto x$ and $\langle x,x^{-1} \rangle \mapsto x^{-1}$ are continuous bijections.
Therefore the mapping $x \mapsto \langle x,x^{-1} \rangle$ is a Borel mapping by the Luzin-Suslin Theorem.  Hence, the
mapping $x \mapsto \langle x,x^{-1} \rangle \mapsto x^{-1}$ is a Borel mapping. Hence, $x \mapsto x^{-1}$, $\mathcal{I}_{R}
\to \mathcal{I}_{R}$, is continuous.
\end{proof}

From Proposition \ref{prop:ContinuityOfInversion} we see that, for any Polish field $\mathbb{F}$, the
multiplicative inversion $x \mapsto x^{-1}$, $\mathbb{F} \setminus \{0\} \to \mathbb{F} \setminus \{0\}$,
is continuous.

\begin{proposition} \label{proposition:M-analytic}
Let $R$ be a Polish ring. If $M \subseteq R$ is a principal one-sided-ideal, then $M$ is an analytic set.
\end{proposition}
\begin{proof}
The left ideal and right ideal cases are similar. For the right-ideal case, let $y \in R$ be the generator for
$M$. That is, $M = \{ yx \ | \  x \in R \}$. Since multiplication is continuous, the mapping $x \mapsto yx$,
$R \to R$, is continuous and $M$ is the image of $R$ under $x \mapsto yx$. Therefore, $M$ is an analytic
set.
\end{proof}

\begin{proposition} \label{proposition:complex-vector-spaces}
Let $X$ be an arbitrary complex topological vector space and $0 \neq x_{0} \in X$. The mapping $\lambda
\mapsto \lambda x_{0}$, $\complexes \to X$, is a homeomorphism onto its range and its range is closed in
$X$.
\end{proposition}
\begin{proof}
This is a special case of the uniqueness theorem for finite dimensional spaces (\cite{kelley-namioka-1976a},
7.3, page 59).
\end{proof}

The ring/field of complex numbers $\complexes$ has $2^{\mathfrak{c}}$ ring/field automorphisms, only two of
which, namely the identity and complex conjugation, are continuous.

\begin{theorem} [\cite{kallman-simmons-1985a}] \label{theorem:bounded-means-continuous}
Any automorphism of $\complexes$ which is bounded on an $F_{\sigma}$ subset of the plane of positive inductive
dimension is necessarily continuous.
\end{theorem}

%%%%%%%%%%%%%%%%%%%%%%%%%%%%%%%%%%%%%%%%%%%%%%%%%%%%%%%%%%%%%%%%%%%%%%%%%%%%%%%%%%%%%%%%%%%%%%%%%%%%%%%%%%%%%%
%%%%%%%%%%%%%%%%%%%%%%%%%%%%%%%%%%%%%%%%%%%%%%%%%%%%%%%%%%%%%%%%%%%%%%%%%%%%%%%%%%%%%%%%%%%%%%%%%%%%%%%%%%%%%%

\section{The Ring of Analytic Functions and Special Cases}  \label{section-roafasc}

Let $\Omega \subseteq \complexes$ be open. $\mathcal{A}(\Omega)$ denotes the collection of complex analytic
functions on $\Omega$ endowed with the compact-open topology, i.e., the topology of uniform convergence on
compact sets.  $\mathcal{A}(\Omega)$ is a commutative Polish $\complexes$-algebra with unity and therefore a
commutative Polish ring with unity with the algebraic operations of point-wise addition and point-wise
multiplication.  If $\lambda \in \complexes$, let $c_{\lambda} \in \mathcal{A}(\Omega)$ be the constant
function taking the value $\lambda$.  The mapping $\lambda \mapsto c_{\lambda}$, $\complexes \to
\mathcal{A}(\Omega)$, is continuous and is therefore a homeomorphism onto its range by Proposition
\ref{proposition:complex-vector-spaces}.  In the future, $\lambda$ will be identified algebraically and topologically
with $c_{\lambda}$.

Several of the following propositions are either well known or elementary and their proof is left to the
reader.

\begin{proposition} \label{proposition:connected-iff-integral-domain}
Let $\Omega$ be open.  Then $\Omega$ is connected if and only if $\mathcal{A}(\Omega)$ is an integral domain.
\end{proposition}

\begin{proposition} \label{proposition:rings-disconnected-open-sets}
Let $\Omega = \bigcup_{n < \kappa} \Omega_{n} \subseteq \complexes$ be open, where each $\Omega_{n}$ is open and
connected, where the $\Omega_{n}$'s are pairwise disjoint and where $\kappa \leq \aleph_{0}$. The rings
$\mathcal{A}(\Omega)$ and $\prod_{n < \kappa} \mathcal{A}(\Omega_n)$ are topologically ring isomorphic.
\end{proposition}

Ideals in $\mathcal{A}(\Omega)$ will play a key role in what follows, especially the following distinguished
class of ideals. For $\alpha \in \Omega$ let
\begin{displaymath}
M_{\alpha} = \{ f \in \mathcal{A}(\Omega) \ | \ f(\alpha) = 0 \}
\end{displaymath}
be the associated ideal.

Note that $f \in \mathcal{A}(\Omega)$ is invertible if and only if $f$ is zero-free. Let $\mathcal{I} = \{ f
\in \mathcal{A}(\Omega) \ | \  f \text{ is zero-free} \}$, an abbreviation of
$\mathcal{I}_{\mathcal{A}(\Omega)}$.

Recall that if $f \in \mathcal{A}(\Omega)$, $\alpha \in \Omega$ and $f(\alpha) = 0$, then there exists a
unique $g \in \mathcal{A}(\Omega)$ so that $f = (z - \alpha)g$. More generally, for any $f \in
\mathcal{A}(\Omega)$ and $\alpha \in \Omega$, there exists a unique $g \in \mathcal{A}(\Omega)$ so that $f =
(z - \alpha)g + f(\alpha)$.

\begin{lemma} \label{lemma:characterization-of-ideals}
$I \subseteq \mathcal{A}(\Omega)$ is a proper principal maximal ideal if and only if there exists $\alpha
\in \Omega$ so that $I = M_{\alpha}$.
\end{lemma}
\begin{proof}
$M_{\alpha}$ is a proper principal ideal since it is generated by $z - \alpha$ and is a maximal ideal since
$\mathcal{A}(\Omega)/M_{\alpha}$ is $\complexes$.

Conversely, let $f$ be the generator for an ideal $I$ in $\mathcal{A}(\Omega)$. If $f$ is zero-free, then $f$
is invertible in $\mathcal{A}(\Omega)$, so the constant $1$ function is in $I$ and $I = \mathcal{A}(\Omega)$
contradicting the fact that $I$ is proper. Hence, there is $\alpha \in \Omega$ for which $f(\alpha) = 0$ and
there is $g \in \mathcal{A}(\Omega)$ such that $f = (z - \alpha)g$.  So
\begin{displaymath}
I = f \mathcal{A}(\Omega) = (z - \alpha) g \mathcal{A}(\Omega) \subseteq (z - \alpha) \mathcal{A}(\Omega) = M_{\alpha}.
\end{displaymath}
$I = M_{\alpha}$ since $I$ was assumed to be maximal.
\end{proof}

\begin{lemma} \label{lem:scalar_characterisation}
Let $\Lambda$ be a countable dense subset of $\complexes$ and $\Omega$ be open and connected. Then, for $f \in
\mathcal{A}(\Omega)$, $f$ is constant if and only
\begin{displaymath}
f \in \bigcap_{\lambda \in \Lambda} (\lambda + (\mathcal{I} \cup \{0\})).
\end{displaymath}
\end{lemma}
\begin{proof}
The implication is trivial if $f$ is constant.

Conversely, suppose $f \in \mathcal{A}(\Omega)$ is non-constant. Then, by the Open Mapping Theorem, there is
some $\alpha \in \Omega$ and some $\lambda \in \Lambda$ so that $f(\alpha) = \lambda$. It follows that $f -
\lambda$ has a zero which means $f - \lambda$ is not invertible and, since $f$ is non-constant, $f - \lambda$
is not identically zero. That is, $f - \lambda \not\in \mathcal{I} \cup \{0\}$.
\end{proof}

\begin{lemma} \label{lemma:rings-some-Borel-sets}
Let $\Omega \subseteq \complexes$ be open and connected, let $R$ be a Polish ring and let $\varphi : R \to
\mathcal{A}(\Omega)$ be an abstract isomorphism of rings.  Then $\varphi^{-1}[\complexes]$,
$\varphi^{-1}[\complexes \setminus \Omega]$ and $\varphi^{-1}[\Omega]$ are Borel sets.
\end{lemma}
\begin{proof}
$\mathcal{I}_{R} = \varphi^{-1}[\mathcal{I}]$ is a Borel set by Proposition \ref{prop:InvertiblesAreBorel}.
\begin{displaymath}
\varphi^{-1}[\complexes] = \bigcap_{\lambda \in \Lambda} (\varphi^{-1}(\lambda) + (\mathcal{I}_{R} \cup \{0\}))
\end{displaymath}
is a countable intersection of Borel sets and so is a Borel set.

Next,
\begin{displaymath}
\complexes \setminus \Omega = \complexes \cap \{ \alpha \in \complexes \ | \ z - \alpha \in \mathcal{I} \} = \complexes \cap (z + \mathcal{I})
\end{displaymath}
since $-\mathcal{I} = \mathcal{I}$.  Therefore
\begin{displaymath}
\varphi^{-1}[\complexes \setminus \Omega] = \varphi^{-1}[\complexes] \cap (\varphi^{-1}(z) + \mathcal{I}_{R})
\end{displaymath}
is a Borel set.

Finally, $\varphi^{-1}[\Omega] = \varphi[\complexes] \setminus \varphi[\complexes \setminus \Omega]$ is a Borel set.
\end{proof}

\begin{lemma} \label{lem:decomposition}
Let $\Omega \subseteq \complexes$ be open and connected, $R$ be a Polish ring and $\varphi : R \to
\mathcal{A}(\Omega)$ be an isomorphism of rings.
Then
\begin{enumerate}[label=(\roman*)]
\item \label{decomposition-closed-sets}
$\varphi^{-1}[\complexes]$ is closed in $R$ and if $M \subseteq R$ is a proper principal maximal ideal, then $M$ is closed;
\item \label{decomposition-the-isomorphism}
for a proper principal maximal ideal $M$,
\begin{displaymath}
\langle x,y \rangle \mapsto x + y, M \oplus \varphi^{-1}[\complexes] \to R,
\end{displaymath}
is an additive group and topological isomorphism.
\end{enumerate}
\end{lemma}
\begin{proof}
\ref{decomposition-closed-sets}
Let $M \subseteq R$ be a proper principal maximal ideal. It follows that there exists $\alpha \in \Omega$ so
that $\varphi[M] = M_{\alpha}$ and $\mathcal{A}(\Omega) \cong M_{\alpha} \oplus \complexes$ as additive
groups. Hence, $R \cong M \oplus \varphi^{-1}[\complexes]$ as additive groups. As $M \cap
\varphi^{-1}[\complexes] = \{0\}$, $M$ is an analytic set by Proposition \ref{proposition:M-analytic}, and
$\varphi^{-1}[\complexes]$ is a Borel set, Corollary \ref{cor:Atim} implies that both
$\varphi^{-1}[\complexes]$ and $M$ are closed in $R$.

\ref{decomposition-the-isomorphism}
Since $\varphi^{-1}[\complexes]$ and $M$ are closed in $R$, they are also additive Polish groups so $M \oplus
\varphi^{-1}[\complexes]$ is an additive Polish group.
The mapping $\langle x,y \rangle \mapsto x + y$, $M \oplus \varphi^{-1}[\complexes] \to R$, is a homemorphism by
Theorem \ref{theorem:becker-kechris} since it is continuous and an additive group isomorphism.
\end{proof}

\begin{lemma} \label{lem:RingAbstractPointEvaluation}
Let $\Omega$ be open and connected and $\varphi : R \to \mathcal{A}(\Omega)$ be a ring isomorphism.
For each $\alpha \in \Omega$, the mapping
\begin{displaymath}
x \mapsto \varphi^{-1}(\varphi(x)(\alpha)), R \to \varphi^{-1}[\complexes],
\end{displaymath}
is continuous.
\end{lemma}
\begin{proof}
For $\alpha \in \Omega$, let $x \mapsto \langle x_{\alpha},x_{\alpha}^{\ast} \rangle$, $R \to \varphi^{-1}[M_{\alpha}]
\oplus \varphi^{-1}[\complexes]$, be the homeomorphic isomorphism of additive groups guaranteed by Lemma
\ref{lem:decomposition}. Since projection is continuous, we see that the mapping $x \mapsto
x_{\alpha}^{\ast}$, $R \to \varphi^{-1}[\complexes]$, is continuous.

Since $x = x_{\alpha} + x_{\alpha}^{\ast}$,
$\varphi(x) = \varphi(x_{\alpha}) + \varphi(x_{\alpha}^{\ast})$,
$\varphi(x)(\alpha) = \varphi(x_{\alpha})(\alpha) + \varphi(x_{\alpha}^{\ast})(\alpha) = 0  + \varphi(x_{\alpha}^{\ast})(\alpha)
= \varphi(x_{\alpha}^{\ast})(\alpha) = \varphi(x_{\alpha}^{\ast})$ as $\varphi(x_{\alpha}) \in M_{\alpha}$
and $\varphi(x_{\alpha}^{\ast}) \in \complexes$.  Therefore $x \mapsto x_{\alpha}^{\ast} = \varphi^{-1}(\varphi(x_{\alpha}^{\ast}))
= \varphi^{-1}(\varphi(x)(\alpha))$ is continuous.
\end{proof}

\begin{theorem} \label{thm:fundamental}
Let $\Omega$ be open and connected. If $R$ is a Polish ring and $\varphi : R \to \mathcal{A}(\Omega)$ is
an abstract ring isomorphism so that $\varphi\restriction_{\varphi^{-1}[\complexes]}$ is
$\mathscr{BP}$-measurable, then $\varphi$ is a homeomorphism.
\end{theorem}
\begin{proof}
$\varphi^{-1}[\complexes]$ is a closed additive subgroup of $R$ by Lemma \ref{lem:decomposition} and therefore
is itself an additive Polish group.  $\varphi\restriction_{\varphi^{-1}[\complexes]}$ is an additive group and
topological isomorphism since it is $\mathscr{BP}$-measurable.  Therefore for each $\alpha \in \Omega$ the
mapping $x \mapsto x_{\alpha}^{\ast} \mapsto \varphi(x_{\alpha}^{\ast}) = \varphi(x)(\alpha)$ is continuous.
Let $\Lambda \subseteq \Omega$ be a countable dense set and let $\psi : \mathcal{A}(\Omega) \to
\prod_{\lambda \in \Lambda} \complexes$ be defined by $\psi(f) = \prod_{\lambda \in \Lambda} f(\lambda)$.
$\psi$ is a continuous injection, its range is a Borel set and $\psi$ is a Borel isomorphism onto its range.
Define $\Theta : R \to \mathcal{A}(\Omega)$ by $\Theta(x) = \psi^{-1}(\prod_{\lambda \in \Lambda}
\varphi(x)(\lambda))$. $\Theta$ is a Borel linear bijection and therefore a topological isomorphism between
$R$ and $\mathcal{A}(\Omega)$.
\end{proof}

\begin{corollary} \label{corollary:c-algebra-case}
Let $A$ be any Polish $\complexes$-algebra and $\varphi : A \to \mathcal{A}(\Omega)$ be an abstract isomorphism of
$\complexes$-algebras, where $\Omega$ is open and connected. Then $\varphi$ is a homeomorphism.
\end{corollary}
\begin{proof}
By Proposition \ref{proposition:complex-vector-spaces}, we have that $\lambda \mapsto \lambda \varphi^{-1}(1)$,
$\complexes \to A$, is a homeomorphism onto its range. $\varphi(\lambda \varphi^{-1}(1)) = \lambda$ since
$\varphi$ is an isomorphism of $\complexes$-algebras.  Hence $\varphi \restriction_{\varphi^{-1}[\complexes]}$
is continuous so Theorem \ref{thm:fundamental} applies to ensure that $\varphi$ is a homeomorphism.
\end{proof}

Let $R$ be a Polish ring.  $R$ is said to be algebraically determined if, given a Polish ring $S$ and an
abstract ring isomorphism, $\varphi : S \to R$, then $\varphi$ is also a topological isomorphism.
$\complexes$ is not an algebraically determined Polish ring since it has many discontinuous automorphisms. On
the other hand $\reals$ is an algebraically determined Polish ring (\cite{kallman-1984a}), a not totally
trivial fact.  Similar definitions can obviously be made for algebraically determined Polish groups, Polish
algebras, Polish Lie rings, or Polish Lie algebras.  The previous corollary states that $\mathcal{A}(\Omega)$
is an algebraically determined $\complexes$-algebra.

\begin{theorem} \label{theorem:rings-exterior-implies-algebraically-determined}
Suppose $\Omega$ is open, connected and $\interior(\Omega^{c}) \ne \emptyset$. Then $\mathcal{A}(\Omega)$ is
an algebraically determined Polish ring. In particular, this means that $\mathcal{A}(\mathbb{D})$ is algebraically
determined as a Polish ring.
\end{theorem}
\begin{proof}
Let $R$ be a Polish ring and let $\varphi : R \to \mathcal{A}(\Omega)$ be an algebraic isomorphism.
$\varphi^{-1}[\complexes]$ is a closed and therefore Polish subring of $R$ by Lemma \ref{lem:decomposition}.
$\varphi^{-1}[\Omega]$ is a Borel set by Lemma \ref{lemma:rings-some-Borel-sets}. Proposition
\ref{prop:ContinuityByExterior} now guarantees that $\varphi\restriction_{\varphi^{-1}[\complexes]}$ is a
Borel mapping. Conclude the proof by applying Theorem \ref{thm:fundamental}.
\end{proof}

\begin{proposition} \label{proposition:prod-ad}
Let $\kappa \leq \aleph_{0}$, suppose $R_{n}$ is an algebraically determined Polish ring with unity for each $n < \kappa$,
and let $R = \prod_{n < \kappa} R_{n}$.  Then $R$ is an algebraically determined Polish ring.
\end{proposition}
\begin{proof}
Let $S$ be a Polish ring and let $\varphi : S \to R$ be an algebraic isomorphism.
For each $n$ let $e_{n} \in R$ be defined by $e_{n}(\ell) = 0_{R_{n}}$ if $\ell \ne n$, $e_{n}(n) = 1_{R_{n}}$ and let $S_{n} = \varphi^{-1}(e_{n})S = S\varphi^{-1}(e_{n})$.
Each $S_{n}$ is an analytic set and subring of $S$.
If $f_{n} = \prod_{\ell < \kappa} 1_{R_{\ell}} - e_{n}$, let $T_{n} = \varphi^{-1}(f_{n})S = S\varphi^{-1}(f_{n})$.
Then $T_{n}$ is an analytic set and subring of $S$, $S = S_{n}T_{n} = T_{n}S_{n}$ and $S_{n} \cap T_{n} = \prod_{n < \kappa} \{0\}$.
Hence, Corollary \ref{cor:Atim} implies that each $S_{n}$ is a closed subring of $S$ and therefore itself is a Polish ring.
Let $\varphi_{n} = \varphi \restriction_{S_{n}}$.
Then $\varphi_{n} : S_{n} \to R_{n}$ is an algebraic isomorphism and therefore is a topological isomorphism since $R_{n}$ is algebraically determined.  Therefore $\prod_{n < \kappa} \varphi_{n} : \prod_{n < \kappa} S_{n} \to \prod_{n < \kappa} R_{n}$ is a topological isomorphism.

The proof will therefore be complete if we prove that $S$ and $\prod_{n < \kappa} S_{n}$ are topologically isomorphic.
Next, note that \(S\) and \(S_{n} \times T_{n}\) are homeomorphic since the natural mapping $S_{n} \times T_{n} \to S$ is a continuous bijection of additive Polish groups.
Thus there is a natural continuous ring bijection $p : S \to \prod_{n < \kappa} S_{n}$ which therefore is a topological isomorphism.
\end{proof}

\begin{corollary} \label{corollary:rings-disconnected-domain}
Suppose $\Omega$ is open and disconnected.  Then $\mathcal{A}(\Omega)$ is an algebraically determined Polish ring.
\end{corollary}
\begin{proof}
Choose $\kappa \leq \aleph_{0}$ so that $\Omega = \bigcup\{ \Omega_{n} \ | \ n < \kappa \}$ where the
$\Omega_{n}$'s are open, connected for $n < \kappa$ and pairwise disjoint. Therefore Proposition
\ref{proposition:rings-disconnected-open-sets} implies that $\mathcal{A}(\Omega)$ and $\prod_{n < \kappa}
\mathcal{A}(\Omega_{n})$ are algebraically and topologically isomorphic. Since $\interior(\Omega_{n}^{c}) \ne
\emptyset$ for each $n$, Theorem \ref{theorem:rings-exterior-implies-algebraically-determined} implies that
each $\mathcal{A}(\Omega_{n})$ is algebraically determined for each $n < \kappa$. Now apply Proposition
\ref{proposition:prod-ad} to conclude the proof.
\end{proof}

The results of this section, in particular Corollary \ref{corollary:c-algebra-case}, Theorem
\ref{theorem:rings-exterior-implies-algebraically-determined}, and Corollary
\ref{corollary:rings-disconnected-domain} may be considered as test cases and exercises in the techniques
developed in this section and are preliminary to the general result in the next section.

%%%%%%%%%%%%%%%%%%%%%%%%%%%%%%%%%%%%%%%%%%%%%%%%%%%%%%%%%%%%%%%%%%%%%%%%%%%%%%%%%%%%%%%%%%%%%%%%%%%%%%%%%%%%%%
%%%%%%%%%%%%%%%%%%%%%%%%%%%%%%%%%%%%%%%%%%%%%%%%%%%%%%%%%%%%%%%%%%%%%%%%%%%%%%%%%%%%%%%%%%%%%%%%%%%%%%%%%%%%%%

\section{The General Case for the Ring of Analytic Functions}

The purpose of this section is to prove that if $\Omega \subseteq \complexes$ is open, then
$\mathcal{A}(\Omega)$ is an algebraically determined Polish ring.  First recall some basic, easily proved
facts about conformal mappings.

$z$, as usual denotes the identity mapping.  Recall that a conformal map between two open subsets $\Omega_{1}$, $\Omega_{2}$ of $\complexes$
is a bijection $\gamma : \Omega_{1} \to \Omega_{2}$ so that both $\gamma$ and $\gamma^{-1}$ are analytic.
A bijective map $\gamma : \Omega_{1} \to \Omega_{2}$ is said to be anti-conformal if $\overline{\gamma} :
\Omega_{1} \to \overline{z}[\Omega_{2}]$ is a conformal mapping. Here, of course, the overline denotes
complex conjugation.  A simple example of a conformal mapping is $f(\zeta) = \alpha \zeta + \beta$, where
$\alpha$, $\beta \in \complexes$ and $\alpha \ne 0$.

\begin{lemma} \label{lem:AntiConformalBijectionBetweenFunctions}
Let $\Omega_{1}$ be connected open and define $\Omega_{2} = \overline{z}[\Omega_{1}]$. Then the map $\varphi :
\mathcal{A}(\Omega_{1}) \to \mathcal{A}(\Omega_{2})$ defined by
\begin{displaymath}
\varphi(f) = \overline{z} \circ f \circ \left(\overline{z}\restriction_{\Omega_{2}} \right)
\end{displaymath}
is a topological isomorphism of Polish rings along with the property that $\varphi(f)' = \varphi(f')$ for each $f \in
\mathcal{A}(\Omega_{1})$.
\end{lemma}

\begin{proposition}  \label{proposition:conformal-ad}
Suppose $\gamma : \Omega_{1} \to \Omega_{2}$ is a conformal bijection where $\Omega_{1}$, $\Omega_{2}$ are both open.
Then $\varphi : \mathcal{A}(\Omega_{1}) \to \mathcal{A}(\Omega_{2})$ defined by $\varphi(f) = f \circ \gamma^{-1}$ is a topological isomorphism of Polish rings.
More generally, $\mathcal{A}(\Omega_{1})$ and $\mathcal{A}(\Omega_{2})$ are topologically isomorphic as Polish rings if $\Omega_{1}$ and $\Omega_{2}$ are open subsets of $\complexes$ along with the property that they are conformally or anti-conformally equivalent.
\end{proposition}

From this we see that deciding whether or not $\mathcal{A}(\Omega)$ is algebraically determined for open and
connected $\Omega \subseteq \complexes$ reduces to deciding it for some conformal or anti-conformal
representative.

\begin{definition}
Let $R$ be a Polish ring and let $S \subseteq R$.  Let
\begin{displaymath}
ls(S) = \{ x \in R  \ | \ x^{n} \in S \text{ for some } n \ge 1 \text{ and } x^{m} \ne x^{n} \text { for all } 1 \le m < n \}
\end{displaymath}
and
\begin{displaymath}
li(S) = \{ x \in R \ | \ x^{n} \in S \text{ for all } n \ge 1 \text{ and } x^{m} \ne x^{n} \text { for all } 1 \le m < n \}.
\end{displaymath}
Comment: the notation is supposed to bring to mind limsup and liminf.
\end{definition}

\begin{lemma} \label{lemma:lsliBorel}
If $R$ is a Polish ring and $S \subseteq R$ is a Borel set (respectively, an analytic set), then
$ls(S)$ and $li(S)$ are both Borel sets (respectively, analytic sets).
\end{lemma}
\begin{proof}
\begin{displaymath}
ls(S) = \left(\bigcup_{n \ge 1} \{ x \in R \ | \ x^{n} \in S \}\right) \cap \left(\bigcup_{1 \le m <n} \{ x \in R \ | \ x^{m} = x^{n} \}\right)^{c}
\end{displaymath}
and
\begin{displaymath}
li(S) = \left(\bigcap_{n \ge 1} \{ x \in R \ | \ x^{n} \in S \}\right) \cap \left(\bigcup_{1 \le m <n} \{ x \in R \ | \ x^{m} = x^{n} \}\right)^{c}
\end{displaymath}
\end{proof}

\begin{theorem}[Weierstrass, {\cite[Theorem 7.32]{burckel}}] \label{thm:Weierstrass}
Let $\Omega \subseteq \complexes$ be open, $A \subseteq \Omega$ be relatively discrete, and $n : A \to
\integers$ be a function with $n(\alpha) \ge 1$ for each $\alpha \in A$.  Then there exists a holomorphic
function $f : \Omega \to \complexes$ so that $f$ has a zero of order $n(\alpha)$ at every $\alpha \in A$ and no
other zeros.
\end{theorem}

\begin{lemma} \label{lemma:ThreeEquivalences}
Let $\closure(\mathbb{D}) \subseteq \Omega \subseteq \complexes$ be open, $\lambda \in li(\Omega)$ and
$\mathcal{A}^{\ast}(\Omega) = \mathcal{A}(\Omega) \setminus \{0\}$. Then the following are equivalent:\\
(i)	$|\lambda| > 1$;\\
(ii) $\{ \lambda^{n} \ | \ n \ge 1 \}$ is relatively discrete in $\Omega$;\\
(iii) there exists $f \in \mathcal{A}^{\ast}(\Omega)$ such that $f(\lambda^{n}) = 0$ for all $n \ge 1$.\\
\end{lemma}
\begin{proof}
(i) $\implies$ (ii).  Suppose that $\lambda \in li(\Omega)$ with $|\lambda| > 1$.
Then an easy induction shows that $|\lambda|^{n} \ge 1 + n(|\lambda| -
1)$ (Bernoulli's Inequality).  Therefore $\{ |\lambda^{n}| \ | \ n \in \naturals \}$ is a sequence of strictly
increasing numbers which tends towards infinity. This guarantees that $\{ \lambda^{n} \ | \ n \in \naturals
\}$ has no accumulation points in $\complexes$ let alone in $\Omega$.

(ii) $\implies$ (iii).  This follows immediately from Theorem \ref{thm:Weierstrass}.

(iii) $\implies$(i).  Suppose $\lambda \in li(\Omega)$ and $|\lambda| \le 1$.   Then the distinct sequence $\{
\lambda^{n} \}_{n \ge 1}$ has a limit point in $\closure(\mathbb{D}) \subseteq \Omega$, which in turn implies
that $f$ is identically zero, a contradiction.
\end{proof}

\begin{theorem} \label{theorem:ring-complete}
If $\Omega \subseteq \complexes$ is open, then $\mathcal{A}(\Omega)$ is an algebraically determined Polish
ring.
\end{theorem}
\begin{proof}
Let $R$ be a Polish ring and let $\varphi : R \to \mathcal{A}(\Omega)$ be an algebraic isomorphism. Our
goal is to prove that $\varphi$ is a topological isomorphism.  We can and do assume that $\Omega$ is connected
by Corollary \ref{corollary:rings-disconnected-domain}.

We can assume that $\closure(\mathbb{D}) \subseteq \Omega \subseteq \complexes$.  If not, apply a simple conformal
mapping of the form $\zeta \mapsto \alpha \zeta + \beta$ and apply Proposition \ref{proposition:conformal-ad}.

$\mathbb{K} = \varphi^{-1}[\complexes]$ is a closed (and therefore Polish) subring of $R$ by Lemma
\ref{lem:decomposition}. To prove the theorem it suffices to show that $\varphi : \mathbb{K} \to
\complexes$ is $\mathscr{BP}$-measurable by Theorem \ref{thm:fundamental}. This will be accomplished by
Proposition \ref{prop:ContinuityByExterior} if we prove that $\varphi^{-1}[\closure(\mathbb{D})] \subseteq
\mathbb{K}$ has the Baire property. We know already that $\varphi^{-1}[\complexes \setminus \Omega] \subseteq
\mathbb{K}$ and $\varphi^{-1}[\Omega] \subseteq \mathbb{K}$ are Borel sets by Lemma
\ref{lemma:rings-some-Borel-sets}. Furthermore $ls(\complexes \setminus \Omega)$ is a Borel set by Lemma
\ref{lemma:lsliBorel}. $(\complexes \setminus \Omega) \cap \closure(\mathbb{D}) = \emptyset$, so if
$\lambda^{n} \in \complexes \setminus \Omega$, then $|\lambda| > 1$ and therefore $ls(\complexes \setminus
\Omega) \cap \closure(\mathbb{D}) = \emptyset$. Notice that $\closure(\mathbb{D})^{c} = ls(\complexes
\setminus \Omega) \cup \{ \lambda \in li(\Omega) \ | \ |\lambda| > 1 \}$. We will be done if we prove that
$\varphi^{-1}[\{ \lambda \in li(\Omega) \ | \ |\lambda| > 1 \}]$ has the Baire property.

Define
\begin{displaymath}
A = \{ \langle x,a,y,b \rangle  \in (R \times \varphi^{-1}[\Omega] \times R \times \mathbb{K}) \ | \  x = (\varphi^{-1}(z) - a) y + b \}
\end{displaymath}
and consider the continuous map $\pi : R \times \varphi^{-1}[\Omega] \times R \times \mathbb{K} \to R
\times \varphi^{-1}[\Omega]$ defined by $\pi(x,a,y,b) = \langle x,a \rangle$.  $A$ is relatively closed in $R \times
\varphi^{-1}[\Omega] \times R \times \mathbb{K}$ and therefore is a Borel subset of $R \times
\varphi^{-1}[\Omega] \times R \times \mathbb{K}$.

We will show that $\pi\restriction_{A}$ is an injective mapping onto $R \times \varphi^{-1}[\Omega]$.  Suppose
\begin{displaymath}
\langle x,a,y_{1},b_{1}\rangle\text{, } \langle x,a,y_{2},b_{2} \rangle \in A.
\end{displaymath}
\noindent
It follows that $(\varphi^{-1}(z) - a) y_{1} + b_{1} = (\varphi^{-1}(z) - a) y_{2} + b_{2}$ necessitating
\begin{displaymath}
(z - \varphi(a))\varphi(y_{1}) + \varphi(b_{1}) = (z - \varphi(a))\varphi(y_{2}) + \varphi(b_{2}).
\end{displaymath}
Since $\varphi(b_{1})$, $\varphi(b_{2}) \in \complexes$, if we evaluate at $\varphi(a)$, we see that
$\varphi(b_{1}) = \varphi(b_{2})$.
Consequently,
\begin{displaymath}
(z - \varphi(a)) \varphi(y_{1}) = (z - \varphi(a)) \varphi(y_{2}) \implies \varphi(y_{1}) = \varphi(y_{2}).
\end{displaymath}
Hence, $y_{1} = y_{2}$ and $b_{1} = b_{2}$ which establishes that $\pi \restriction_{A}$ is an injection.

To see that $\pi \restriction_{A}$ is a surjection onto $R \times \varphi^{-1}[\Omega]$,
let $\langle x,a \rangle \in R \times \varphi^{-1}[\Omega]$.  Let $f = \varphi(x)$ and $\varphi(a) = \alpha$, then there
exists $g \in \mathcal{A}(\Omega)$ so that $f = (z - \alpha)g + f(\alpha)$.  If $y = \varphi^{-1}(g)$ and $b =
\varphi^{-1}(f(\alpha))$, then $x = (\varphi^{-1}(z) - a)y + b$, $\langle x,a,y,b \rangle \in A$ and $\pi(x,a,y,b) = \langle x,a \rangle$.

Now, as $\pi \restriction_{A}$ is a continuous bijection onto $R \times \varphi^{-1}[\Omega]$, let $\pi^{\ast}
: R \times \varphi^{-1}[\Omega] \to R \times \varphi^{-1}[\Omega] \times R \times \mathbb{K}$ be the
Borel mapping $\pi^{\ast} = (\pi\restriction_{A})^{-1}$. Let $p: R \times \varphi^{-1}[\Omega] \times R \times
\mathbb{K} \to \mathbb{K}$ be the projection onto the fourth coordinate, $p(x,a,y,b) = b$.

For each $n \ge 1$, notice that $\langle x,a \rangle \mapsto \langle x,a^{n} \rangle$, $R \times \mathbb{K} \to R \times \mathbb{K}$,
is continuous. Lemma \ref{lemma:lsliBorel} implies that $li(\varphi^{-1}[\Omega])$ is a Borel set since
$\varphi^{-1}[\Omega]$ is a Borel set. Note that $a^{n} \in li(\varphi^{-1}[\Omega])$ for all $n \ge 1$ if $a
\in li(\varphi^{-1}[\Omega])$.

With all this in hand, we now define
\begin{displaymath}
B = \{ \langle x,a \rangle \in R^{\ast} \times li(\varphi^{-1}[\Omega]) \ | \ (p \circ \pi^{\ast})(x,a^{n}) = 0 \text{ for all } n \ge 1 \}
\end{displaymath}
where $R^{\ast} = R \setminus \{0\}$. Notice that $B$ is a Borel subset of $R \times li(\varphi^{-1}[\Omega])$
as $R^{\ast} \times \varphi^{-1}[\Omega]$ is a Borel subset of $R \times \mathbb{K}$ and $\langle x,a \rangle \mapsto (p
\circ \pi^{\ast})(x,a^{n})$, $R \times \varphi^{-1}[\Omega] \to \mathbb{K}$, is a Borel mapping
for each $n \ge 1$.  Hence, letting $P: R \times \varphi^{-1}[\Omega] \to \varphi^{-1}[\Omega]$ be the
projection onto the second coordinate, we see that $P[B]$ is an analytic set.

The next step is to see that
\begin{displaymath}
P[B] = \varphi^{-1}[\{ \lambda \in li(\Omega) \ | \ |\lambda| > 1 \}].
\end{displaymath}
\noindent
Toward this end, we show that $(p \circ \pi^{\ast})(x,a) = \varphi^{-1}(\varphi(x)(\varphi(a)))$. Let
$\langle x,a,y,b \rangle = \pi^{\ast}(x,a)$ and notice that
\begin{displaymath}
\varphi(x) = (z - \varphi(a))\varphi(y) + \varphi(b) \implies \varphi(x)(\varphi(a)) = \varphi(b) \implies b = \varphi^{-1}(\varphi(x)(\varphi(a))).
\end{displaymath}
This establishes that $p \circ \pi^{\ast}(x,a) = \varphi^{-1}(\varphi(x)(\varphi(a)))$.

Notice that $\varphi^{-1}[li(\Omega)] = li(\varphi^{-1}[\Omega])$.  Now suppose $a \in P[B]$ and let $x \in
R^{\ast}$ be so that $\langle x,a \rangle \in B$. By our definition of $B$, $a \in li(\varphi^{-1}[\Omega])$ which, by our
initial observation, implies that $\varphi(a) \in li(\Omega)$.  By virtue of $\langle x,a \rangle \in B$, we have that
\begin{align*}
0 = (p \circ \pi^{\ast})(x,a^{n}) = \varphi^{-1}(\varphi(x)(\varphi(a^{n}))) &= \varphi^{-1}(\varphi(x)(\varphi(a)^{n}))\\
&\implies \varphi(x)(\varphi(a)^{n}) = 0
\end{align*}
for each $n \ge 1$. Hence, since $x \neq 0$ necessitates $\varphi(x) \neq 0$, we apply Lemma
\ref{lemma:ThreeEquivalences} to conclude that $|\varphi(a)| > 1$.  Therefore $a \in P[B]$ implies that $a \in
\varphi^{-1}[\{ \lambda \in li(\Omega) \ | \ |\lambda| > 1 \}]$ and $P[B] \subseteq \varphi^{-1}[\{ \lambda \in
li(\Omega) \ | \ |\lambda| > 1 \}]$.

On the other hand suppose that $a \in \varphi^{-1}[\{ \lambda \in li(\Omega) \ | \ |\lambda| > 1 \}]$ then
$\varphi(a) \in li(\Omega)$.  Pick a non-zero $f \in \mathcal{A}(\Omega)$ so that $f(\varphi(a)^{n}) = 0$ for
each $n \ge 1$ by Theorem \ref{thm:Weierstrass}.  Hence, $\langle \varphi^{-1}(f),a \rangle \in B$, so $a \in P[B]$,
$\varphi^{-1}[\{ \lambda \in li(\Omega) \ | \ |\lambda| > 1 \}] \subseteq P[B]$, $P[B] = \varphi^{-1}[\{ \lambda
\in li(\Omega) \ | \ |\lambda| > 1 \}]$ and $\varphi^{-1}[\{ \lambda \in li(\Omega) \ | \ |\lambda| > 1 \}]$
is an analytic subset of $\mathbb{K}$.
\end{proof}

%%%%%%%%%%%%%%%%%%%%%%%%%%%%%%%%%%%%%%%%%%%%%%%%%%%%%%%%%%%%%%%%%%%%%%%%%%%%%%%%%%%%%%%%%%%%%%%%%%%%%%%%%%%%%%
%%%%%%%%%%%%%%%%%%%%%%%%%%%%%%%%%%%%%%%%%%%%%%%%%%%%%%%%%%%%%%%%%%%%%%%%%%%%%%%%%%%%%%%%%%%%%%%%%%%%%%%%%%%%%%

\section{A Theorem of Bers}

The purpose of this section is to prove by general principles a theorem of Bers \cite{bers-1948a} as an easy
application of Theorem \ref{theorem:ring-complete}.

\begin{proposition}
Suppose $\gamma : \Omega_{1} \to \Omega_{2}$ is a conformal mapping where $\Omega_{1}$ and $\Omega_{2}$ are
open and $\Omega_{1}$ is connected.
Then $\varphi : \mathcal{A}(\Omega_{1}) \to \mathcal{A}(\Omega_{2})$ defined by $\varphi(f) = f \circ
\gamma^{-1}$ is a topological isomorphism of rings.  The same is true if $\gamma$ is anti-conformal.
\end{proposition}

\begin{lemma} \label{lem:image-of-scalars}
Let $\varphi : \mathcal{A}(\Omega_{1}) \to \mathcal{A}(\Omega_{2})$ be a ring isomorphism where $\Omega_{1}$
and $\Omega_{2}$ are connected.  Then $\varphi[\complexes] = \complexes$ and $\varphi$ is either the identity
or complex conjugation on $\complexes$.
\end{lemma}
\begin{proof}
$\varphi(1) = 1$, therefore $\varphi(m) = m$ for $m \in \integers$ and furthermore $\varphi(r) = r$ for all $r
\in \rationals$.   $\varphi$ is continuous by Theorem \ref{theorem:ring-complete} and therefore $\varphi(x) =
x$ for all $x \in \reals$.  $\varphi(i)^{2} = \varphi(i^{2}) = \varphi(-1) = -1$ so $\varphi(i) = \pm i$.
Hence $\varphi[\complexes] = \complexes$ and $\varphi$ behaves either like the identity or complex conjugation
on $\complexes$.
\end{proof}

\begin{lemma} \label{lem:producing_a_bijection_between_domains}
Let $\varphi : \mathcal{A}(\Omega_{1}) \to \mathcal{A}(\Omega_{2})$ be an abstract ring isomorphism where
$\Omega_{1}$ and $\Omega_{2}$ are open and $\Omega_{1}$ is connected. Then there is a bijection $\gamma :
\Omega_{1} \to \Omega_{2}$ such that $\varphi(f(\alpha)) = \varphi(f)(\gamma(\alpha))$ for any $f \in
\mathcal{A}(\Omega_{1})$ and $\alpha \in \Omega_{1}$.
\end{lemma}
\begin{proof}
For $\alpha \in \Omega_{1}$ and $\beta \in \Omega_{2}$, let $M_{\alpha} = \{ f \in \mathcal{A}(\Omega_{1}) \ |
\ f(\alpha) = 0\}$ and $M^{\ast}_{\beta} = \{ f \in \mathcal{A}(\Omega_{2}) \ | \ f(\beta) = 0 \}$.
Now, let $\alpha \in \Omega_{1}$ be arbitrary and choose $\gamma(\alpha) \in \Omega_{2}$, using Lemma
\ref{lemma:characterization-of-ideals}, so that
\begin{displaymath}
M^{\ast}_{\gamma(\alpha)} = \varphi[M_{\alpha}].
\end{displaymath}
$\gamma : \Omega_{1} \to \Omega_{2}$ is a bijection since $\varphi$ is an isomorphism of rings.

Notice that $f - f(\alpha) \in M_{\alpha}$. Then $\varphi(f) - \varphi(f(\alpha)) \in
M^{\ast}_{\gamma(\alpha)}$. Since $\varphi(f(\alpha)) \in \complexes$ by Lemma \ref{lem:image-of-scalars}, we see
that $\varphi(f)(\gamma(\alpha)) - \varphi(f(\alpha)) = 0$. That is, $\varphi(f(\alpha)) =
\varphi(f)(\gamma(\alpha))$.
\end{proof}

\begin{theorem}[Bers \cite{bers-1948a}] \label{thm:Bers}
Let $\varphi : \mathcal{A}(\Omega_{1}) \to \mathcal{A}(\Omega_{2})$ be an abstract ring isomorphism where $\Omega_{1}$ is open and connected.
Then there exists a conformal or anti-conformal mapping $\gamma : \Omega_{1} \to \Omega_{2}$ so that
\begin{displaymath}
\varphi(f) = f \circ \gamma^{-1} \text{ or } \varphi(f) = \overline{f \circ \gamma^{-1}}.
\end{displaymath}
\end{theorem}
\begin{proof}
$\varphi$ is a homeomorphism by Theorem \ref{theorem:ring-complete}.  Let $\gamma : \Omega_{1} \to \Omega_{2}$ be the
bijection provided by Lemma \ref{lem:producing_a_bijection_between_domains}. We proceed now by cases.

\textsc{Case I.}
Suppose $\varphi \restriction_{\complexes} = z$.
Then, note that, for $\alpha \in \Omega_{1}$,
\begin{displaymath}
\alpha = z(\alpha) = \varphi(z(\alpha)) = \varphi(z)(\gamma(\alpha)).
\end{displaymath}
Hence, $\varphi(z) \circ \gamma = z \restriction_{\Omega_{1}}$. Since $\gamma$ is a bijection, we see that
$\gamma^{-1} = \varphi(z)$ so $\gamma$ is a conformal mapping.

Now, for any $f \in \mathcal{A}(\Omega_{1})$ and $\alpha \in \Omega_{1}$, notice that
\begin{displaymath}
f(\alpha) = \varphi(f(\alpha)) = \varphi(f)(\gamma(\alpha)) \implies f = \varphi(f) \circ \gamma \implies \varphi(f) = f \circ \gamma^{-1}.
\end{displaymath}

\textsc{Case II.}
Suppose $\varphi \restriction_{\complexes} = \overline{z}$.
Then, note that, for $\alpha \in \Omega_{1}$,
\begin{displaymath}
\overline{\alpha} = \overline{z(\alpha)} = \varphi(z(\alpha)) = \varphi(z)(\gamma(\alpha)).
\end{displaymath}
Hence, $\varphi(z) \circ \gamma = \overline{z}\restriction_{\Omega_{1}}$.
It follows that $\overline{\gamma^{-1}} = \varphi(z)$ so $\gamma$ is an anti-conformal mapping.

Now, for any $f \in \mathcal{A}(\Omega_{1})$ and $\alpha \in \Omega_{1}$, notice that
\begin{align*}
	\overline{f(\alpha)} = \varphi(f(\alpha)) = \varphi(f)(\gamma(\alpha)) &\implies \overline{z} \circ f = \varphi(f) \circ \gamma\\
	&\implies \overline{z} \circ f \circ \gamma^{-1} = \varphi(f)\\
	&\implies \varphi(f) = \overline{f \circ \gamma^{-1}}.
\end{align*}
\end{proof}

%%%%%%%%%%%%%%%%%%%%%%%%%%%%%%%%%%%%%%%%%%%%%%%%%%%%%%%%%%%%%%%%%%%%%%%%%%%%%%%%%%%%%%%%%%%%%%%%%%%%%%%%%%%%%%
%%%%%%%%%%%%%%%%%%%%%%%%%%%%%%%%%%%%%%%%%%%%%%%%%%%%%%%%%%%%%%%%%%%%%%%%%%%%%%%%%%%%%%%%%%%%%%%%%%%%%%%%%%%%%%

\section{The Bounded Analytic Functions on the Disk}

Recall that Liouville's Theorem informs us that the only bounded entire functions are the constants. So in
this trivial case the bounded entire functions form a Polish ring. But what about the bounded analytic
functions on other domains?  The results of this section should be compared with the classic results of Kakutani
(\cite{kakutani-1955a} and \cite{kakutani-1957a}).

Let $B(\mathbb{D})$ be the abstract ring of bounded analytic functions on $\mathbb{D}$. Let $H^{\infty}$ be
the ring $B(\mathbb{D})$ endowed with the topology of uniform convergence (the topology compatible with the
sup norm metric) and identify each scalar $\lambda \in \complexes$ with the constant function taking value
$\lambda$. In this section we assume that the abstract ring $B(\mathbb{D})$ is given a fixed Polish ring
topology and we will show that this leads to a contradiction.

The following proposition must be well known.  A proof is hinted at in
https://math.stackexchange.com/questions/1689215/h-infty-is-not-separable.

\begin{proposition} \label{prop:SeparabilityOfHinf}
The space $H^{\infty}$ is complete metrizable but not separable.
\end{proposition}

For a hint of the proof, for each $\lambda \in S^{1}$, let
\begin{displaymath}
f_{\lambda} = \exp\left( \frac{z + \lambda}{z - \lambda} \right).
\end{displaymath}
Then show that $\|f_{\lambda}\| \le 1$ and $\|f_{\lambda_{1}} - f_{\lambda_{2}}\| \ge 1$ for all
$\lambda_{1}$, $\lambda_{2} \in S^{1}$ with $\lambda_{1} \ne \lambda_{2}$.

\begin{lemma} \label{lem:BoundedDecomposition}
For any $f \in B(\mathbb{D})$ and $\alpha \in \mathbb{D}$, there exists $g \in B(\mathbb{D})$ so that $f = (z
- \alpha)g + f(\alpha)$. Consequently, $M_{\alpha} = \{ f\in B(\mathbb{D}) \ | \ f(\alpha) = 0 \}$ is a
principal maximal ideal.
\end{lemma}
\begin{proof}
We need only check the statement for non-constant functions as the statement obviously holds for constant functions ($g$ is taken to be zero when $f$ is constant).
Define $g : \mathbb{D} \to \complexes$ by the rule
\begin{displaymath}
g(\zeta) =
    \begin{cases}
		\dfrac{f(\zeta) - f(\alpha)}{\zeta - \alpha}, & \zeta \neq \alpha;\\
		f'(\alpha), & \zeta = \alpha.
		\end{cases}
\end{displaymath}
$g$ is analytic in $\mathbb{D}$ so we need only check it is bounded.
Let $U$ be a connected open set so that
\begin{displaymath}
\alpha \in U \subseteq \text{cl}(U) \subseteq \mathbb{D}
\end{displaymath}
and, for each $\zeta \in U$,
\begin{displaymath}
\left| \frac{f(\zeta) - f(\alpha)}{\zeta - \alpha} \right| < |f'(\alpha)| + 1.
\end{displaymath}
Let $D = \min\{ |\zeta - \alpha| \ | \ \zeta \in \mathbb{D} \setminus U \}$ and notice that $D > 0$ since $U$
is an open set containing $\alpha$.
Now, for any $\zeta \in \mathbb{Z}$,
\begin{displaymath}
|g(\zeta)| \leq \max \left\{ |f'(\alpha)| + 1, \frac{2\|f\|_\infty}{D} \right\}.
\end{displaymath}
and therefore $g$ is bounded.

Lastly, as the ring homomorphism $f \mapsto f(\alpha)$, $B(\mathbb{D}) \to \complexes$, has kernel
$M_{\alpha}$, we see that $M_{\alpha}$ is a principal maximal ideal in $B(D)$.
\end{proof}

The next lemma finds inspiration from a comment in the proof of \cite[Proposition 3]{RoydenMeromorphic} where
H. L. Royden suggests an algebraic characterization of the constant functions in any ring consisting of
meromorphic functions.

\begin{lemma} \label{lem:BoundedScalars}
Let $\varphi : B(\mathbb{D}) \to H^{\infty}$ be the identity map.
Then $\varphi^{-1}[\complexes]$ is an analytic set.
\end{lemma}
\begin{proof}
Let $\Lambda$ be a countable dense subset of $\complexes$.
First, we will see that \(f\) is constant if and only if, for all \(\lambda \in \Lambda\), there exists \(g \in B(\mathbb{D})\) such that \(g^{2} = f - \lambda\).

If $f$ is a constant function, then so is $f - \lambda$ and $f - \lambda$ has a square root.

Now suppose $f$ is non-constant and pick $\alpha \in \mathbb{D}$ so that $f'(\alpha) \neq 0$. Pick $r > 0$ so
that for all $\zeta \in \closure(B(\alpha,r))$, $f'(\zeta) \neq 0$. Since $f$ is analytic, $f[B(\alpha,r)]$ is
open so pick $\lambda \in f[B(\alpha,r)] \cap \Lambda$ and $\zeta \in B(\alpha,r)$ so that $f(\zeta) =
\lambda$. Suppose we have $g \in B(\mathbb{D})$ so that $g^{2} = f - \lambda$. Notice that $g^{2}(\zeta) =
f(\zeta) - \lambda = 0$ and $g(\zeta) = 0$. Next $2gg' = f'$ which gives
\begin{displaymath}
2g(\zeta)g'(\zeta) = f'(\zeta) \neq 0,
\end{displaymath}
contradicting the fact that $g(\zeta) = 0$. So $f -\lambda$ has no square root.

Now, notice that
\begin{displaymath}
A_{\lambda} = \{ \langle f, g \rangle \in B(\mathbb{D})^{2} \ | \ g^{2} = f - \lambda \}
\end{displaymath}
is closed so the projection $\pi: B(\mathbb{D})^{2} \to B(\mathbb{D})$ onto the first coordinate
guarantees that $\pi[A_{\lambda}]$ is an analytic set. Finally, analytic sets are closed under countable
intersections so
\begin{displaymath}
A = \bigcap\{ \pi[A_{\lambda}] \ | \ \lambda \in \Lambda \}
\end{displaymath}
is an analytic set. We are done since $A = \varphi^{-1}[\complexes]$.
\end{proof}

\begin{lemma} \label{lem:BoundedEvaluation}
Let $\varphi : B(\mathbb{D}) \to H^{\infty}$ be the identity map and let $\alpha \in \mathbb{D}$.
Then
\begin{enumerate}[label=(\roman*)]
\item \label{BoundedRing_MaximalIdealsClosed}
$M_{\alpha}$ is closed in $B(\mathbb{D})$;
\item \label{BoundedRing_ScalarsClosed}
$\varphi^{-1}[\complexes]$ is closed in $B(\mathbb{D})$;
\item \label{BoundedRing_DecompositionIsHomeomorphism}
$f \mapsto f(\alpha)$, $B(\mathbb{D}) \to \varphi^{-1}[\complexes]$, is continuous.
\end{enumerate}
\end{lemma}
\begin{proof}
\ref{BoundedRing_MaximalIdealsClosed}-\ref{BoundedRing_ScalarsClosed}
Lemma \ref{lem:BoundedDecomposition} implies that $M_{\alpha}$ is a principal ideal in
$B(\mathbb{D})$, so $M_{\alpha}$ is an analytic set.   Since Lemma \ref{lem:BoundedScalars}
establishes that $\varphi^{-1}[\complexes]$ is an analytic additive subgroup of $B(\mathbb{D})$, both
$M_{\alpha}$ and $\varphi^{-1}[\complexes]$ are closed subsets of $B(\mathbb{D})$ by an application of Lemma
\ref{cor:Atim}.

\ref{BoundedRing_DecompositionIsHomeomorphism}
The natural decomposition $f \mapsto \langle f - f(\alpha), f(\alpha) \rangle$, $B(\mathbb{D}) \to
M_{\alpha} \oplus \varphi^{-1}[\complexes]$, is a homeomorphism qua a continuous group isomorphism
between additive Polish groups. Lastly, $f \mapsto f(\alpha)$, $B(\mathbb{D}) \to
\varphi^{-1}[\complexes]$, is continuous.
\end{proof}

\begin{theorem} \label{thm:NoPolishRingTopologyOnBounded}
The abstract ring of bounded analytic functions on the disk cannot be made into a Polish ring.
\end{theorem}
\begin{proof}
Let $F_{n} = \{ f \in B(\mathbb{D}) \ | \ \|f\|_\infty \leq n \}$.  Since the mapping $B(\mathbb{D}) \to
\complexes$, $f \mapsto f(\alpha)$, is continuous for each $\alpha \in \mathbb{D}$, each $F_{n}$ is closed in
$B(\mathbb{D})$.  Since $B(\mathbb{D}) = \bigcup_{n \ge 1} F_{n}$, the Baire Category Theorem implies that some
$F_{n}$, say $F_{m}$, contains a nonempty open subset $U \subseteq F_{m} \subseteq
B(\mathbb{D})$.   But then $0 \in V = U - U \subseteq F_{m} - F_{m} \subseteq F_{2m}$, where $V$ is now an
open neighborhood of $0 \in B(\mathbb{D})$.   Furthermore, if $\mu > 0$, then $0 \in \mu \cdot V \subseteq \mu
\cdot F_{2m} = F_{\mu 2m}$, where $\mu \cdot V$ is an open neighborhood of $0 \in B(\mathbb{D})$ since
multiplication by $\mu$ is a homeomorphism in $B(\mathbb{D})$.  From this it is elementary to check that every
open subset of $H^{\infty}$ is open in $B(\mathbb{D})$, i.e., the identity mapping $B(\mathbb{D}) \to
H^{\infty}$ is continuous. But the continuous image of a separable space is separable, contradicting
Proposition \ref{prop:SeparabilityOfHinf}.
\end{proof}

%%%%%%%%%%%%%%%%%%%%%%%%%%%%%%%%%%%%%%%%%%%%%%%%%%%%%%%%%%%%%%%%%%%%%%%%%%%%%%%%%%%%%%%%%%%%%%%%%%%%%%%%%%%%%%
%%%%%%%%%%%%%%%%%%%%%%%%%%%%%%%%%%%%%%%%%%%%%%%%%%%%%%%%%%%%%%%%%%%%%%%%%%%%%%%%%%%%%%%%%%%%%%%%%%%%%%%%%%%%%%

\section{The Field of Meromorphic Functions}

In what follows $\mathcal{M}(\Omega)$ denotes the abstract field of meromorphic functions on $\Omega$.  The
result of this section is presented for its own interest and as a complement to Theorem
\ref{theorem:ring-complete}.

\begin{theorem} \label{meromorphic-not-Polish}
The abstract field $\mathcal{M}(\Omega)$ cannot be made into a Polish field.
\end{theorem}
\begin{proof}
Suppose that $M(\Omega)$ can be given a Polish field topology.  Then $G = M(\Omega) \setminus {0}$ may be
viewed as a multiplicative Polish group.  For each $\alpha \in \Omega$ let $o(f,\alpha) \in \integers$ be the
order (positive for zeros, negative for poles, and $0$ otherwise) of $f$ at $\alpha$.  Each of the mappings $f
\mapsto o(f,\alpha)$, $G \to (\integers,+)$, is a homomorphism of groups which is continuous by a theorem of
Dudley \cite{dudley-1961a}.  Since arbitrary subsets of the discrete space $\integers$ are both open and
closed, then $\mathcal{A}(\Omega) \setminus \{0\}  = \bigcap_{\alpha \in \Omega} \{ f \in G \ | \ o(f,\alpha)
\ge 0 \}$  is closed in $G$. Therefore $\mathcal{A}(\Omega)$ is a closed subset and hence a Polish subring of
$\mathcal{M}(\Omega)$. Theorem \ref{theorem:ring-complete} implies that $\mathcal{A}(\Omega) \subseteq
\mathcal{M}(\Omega)$ has its usual topology.  Then $z^{2} + z/n \to z^{2}$ in the usual topology, but $1 =
o(z^{2} + z/n,0) \not \to o(z^{2},0) = 2$, a contradiction.
\end{proof}

%%%%%%%%%%%%%%%%%%%%%%%%%%%%%%%%%%%%%%%%%%%%%%%%%%%%%%%%%%%%%%%%%%%%%%%%%%%%%%%%%%%%%%%%%%%%%%%%%%%%%%%%%%%%%%
%%%%%%%%%%%%%%%%%%%%%%%%%%%%%%%%%%%%%%%%%%%%%%%%%%%%%%%%%%%%%%%%%%%%%%%%%%%%%%%%%%%%%%%%%%%%%%%%%%%%%%%%%%%%%%
%%%%%%%%%%%%%%%%%%%%%%%%%%%%%%%%%%%%%%%%%%%%%%%%%%%%%%%%%%%%%%%%%%%%%%%%%%%%%%%%%%%%%%%%%%%%%%%%%%%%%%%%%%%%%%
%%%%%%%%%%%%%%%%%%%%%%%%%%%%%%%%%%%%%%%%%%%%%%%%%%%%%%%%%%%%%%%%%%%%%%%%%%%%%%%%%%%%%%%%%%%%%%%%%%%%%%%%%%%%%%
%%%%%%%%%%%%%%%%%%%%%%%%%%%%%%%%%%%%%%%%%%%%%%%%%%%%%%%%%%%%%%%%%%%%%%%%%%%%%%%%%%%%%%%%%%%%%%%%%%%%%%%%%%%%%%
%%%%%%%%%%%%%%%%%%%%%%%%%%%%%%%%%%%%%%%%%%%%%%%%%%%%%%%%%%%%%%%%%%%%%%%%%%%%%%%%%%%%%%%%%%%%%%%%%%%%%%%%%%%%%%
%%%%%%%%%%%%%%%%%%%%%%%%%%%%%%%%%%%%%%%%%%%%%%%%%%%%%%%%%%%%%%%%%%%%%%%%%%%%%%%%%%%%%%%%%%%%%%%%%%%%%%%%%%%%%%

\bibliographystyle{plain}

%%%%%%%%%%%%%%%%%%%%%%%%%%%%%%%%%%%%%%%%%%%%%%%%%%%%%%%%%%%%%%%%%%%%%%%%%%%%%%%%%%%%%%%%%%%%%%%%%%%%%%%%%%%%%%
%%%%%%%%%%%%%%%%%%%%%%%%%%%%%%%%%%%%%%%%%%%%%%%%%%%%%%%%%%%%%%%%%%%%%%%%%%%%%%%%%%%%%%%%%%%%%%%%%%%%%%%%%%%%%%

\end{document}